\documentclass[reqno]{amsart}
\usepackage{amssymb}

\newcommand{\qform}[1]{{\left\langle{#1}\right\rangle}}
\newcommand{\pform}[1]{{\langle\!\langle{#1}\rangle\!\rangle}}
\newcommand{\binv}{\overline{\rule{2.5mm}{0mm}\rule{0mm}{4pt}}}
\newcommand{\Qc}{(Q,\binv)}
\renewcommand{\H}{\mathbb{H}}

\DeclareMathOperator{\Ad}{Ad}
\DeclareMathOperator{\ad}{ad}
\DeclareMathOperator{\Id}{Id}
\newcommand{\id}{\Id}

\DeclareMathOperator{\disc}{disc}
\DeclareMathOperator{\End}{End}
\DeclareMathOperator{\Sym}{Sym}
\DeclareMathOperator{\Nrp}{Nrp}
\DeclareMathOperator{\Nrd}{Nrd}
\DeclareMathOperator{\Int}{Int}
\DeclareMathOperator{\ind}{ind}
\DeclareMathOperator{\gr}{gr}

\newtheorem{prop}{Proposition}[section]
\newtheorem{lem}[prop]{Lemma}
\newtheorem{thm}[prop]{Theorem}
\newtheorem{cor}[prop]{Corollary}

\title[Hyperbolicity over the function field of a conic]{Algebras with
  involution that become hyperbolic over the function field of a
  conic}
\author{Anne Qu\'eguiner-Mathieu} 
\address{Universit\'e Paris 13 (LAGA)\\
CNRS (UMR 7539)\\
Universit\'e Paris 12 (IUFM)\\
93430 Villetaneuse, France}
\email{queguin@math.univ-paris13.fr}
\urladdr{http://www-math.univ-paris13.fr/{\textasciitilde}queguin/}

 \author{Jean-Pierre Tignol}
\address{D\'epartement de Math\'ematiques\\ 
Universit\'e catholique de Louvain\\
Chemin du cyclotron, 2\\
B1348 Louvain-la-Neuve\\
Belgique}
\email{jean-pierre.tignol@uclouvain.be}
\urladdr{http://wwww.math.ucl.ac.be/membres/tignol}

\thanks{The second author gratefully
  acknowledges the hospitality of Universit\'e Paris~13, where part of
the work leading to this paper was carried out. He was supported in
part by the F.R.S.--FNRS}

\begin{document}

\begin{abstract}
We study central simple algebras with involution of the first kind
that become hyperbolic over the function field of the conic associated
to a given quaternion algebra $Q$. We classify these algebras in
degree~$4$ and give an example of such a division algebra with
orthogonal involution of degree~$8$ that does not contain $\Qc$, even
though it contains $Q$ and is totally decomposable into a tensor
product of quaternion algebras.
\end{abstract}

\maketitle

Given two central simple algebras with involution $(A,\sigma)$ and
$(B,\tau)$ over a field $F$, we say that $(A,\sigma)$ contains
$(B,\tau)$ if $A$ contains a $\sigma$-stable subalgebra isomorphic to
$B$ over which the involution induced by $\sigma$ is conjugate to
$\tau$. By the double centralizer theorem~\cite[(1.5)]{KMRT}, this is
also equivalent to saying that $(A,\sigma)$ is isomorphic to a tensor
product $(A,\sigma)\simeq(B,\tau)\otimes(C,\gamma)$ for some 
central simple algebra with involution $(C,\gamma)$ over $F$. 

Let $\Qc$ be a quaternion division algebra over $F$, endowed with its
canonical involution. We denote by $F_Q$ the function field of the
associated conic, which is the Severi--Brauer variety of $Q$. 
Since $\binv$ is of symplectic type, it becomes hyperbolic over any
field that splits $Q$, hence in particular over $F_Q$. From this, one
may easily deduce that \emph{any $(A,\sigma)$ that contains $\Qc$ becomes
  hyperbolic over $F_Q$.} The main theme of this paper is to
investigate the reverse implication. In the case where $A$ is split
and $\sigma$ is anisotropic, it is an easy consequence of the
Cassels--Pfister subform theorem in the algebraic theory of quadratic
forms that the converse holds, see Proposition~\ref{split.prop}. When
$A$ is not split, the problem is much more delicate, and comparable to
the characterization of quadratic forms that become \emph{isotropic}
over $F_Q$, which was studied by Hoffmann, Lewis, and Van Geel
\cite{Hoff:thesis}, \cite{HLVG}, \cite{HVG}. Using an example from
\cite{HVG} of a $7$-dimensional quadratic form over a suitable field
$F$ that becomes isotropic over $F_Q$, we construct in
\S\ref{totdec.sec} a division algebra $A$ of degree~$8$ with an
orthogonal involution $\sigma$ such that $(A,\sigma)$ becomes
hyperbolic over $F_Q$ but does not contain $\Qc$, even though $A$
contains $Q$ and $(A,\sigma)$ decomposes into a tensor product of
quaternion algebras with involution. This situation does not occur in
lower degrees.

Algebras with involution that become split hyperbolic over $F_Q$ are
considered in \S\ref{split.sec}. It is shown in
Proposition~\ref{split.cor} that their anisotropic kernel contains
$\Qc$ (if it is not trivial). This applies in particular to algebras
of degree~$2m$ with $m$ odd, see Corollary~\ref{split2.cor}. The case
of algebras of degree~$4$ is completely elucidated in
\S\ref{deg4.sec}, using Clifford algebras for orthogonal involutions
and a relative cohomological invariant of degree~$3$ due to
Knus--Lam--Shapiro--Tignol for symplectic involutions. In the
symplectic case, we classify the algebras of degree~$4$ that become
hyperbolic over $F_Q$ but do not contain $\Qc$, see
Theorem~\ref{deg4symp.thm}. We show in \S\ref{min.sec} that our result
is equivalent to the Hoffmann--Lewis--Van Geel classification of
$5$-dimensional quadratic forms that become isotropic over $F_Q$
without containing a Pfister neighbour of the norm form of $Q$, see
Corollary~\ref{qf.cor}. Orthogonal involutions on algebras of
degree~$8$ are considered in \S\S\ref{totdec.sec} and \ref{deg8.sec},
using triality. In \S\ref{totdec.sec} we relate tensor products of
quaternion algebras to quadratic forms of dimension~$8$ with trivial
discriminant. The algebras with involution that do not decompose into
tensor products of quaternion algebras with involution and become
hyperbolic over $F_Q$ are determined in \S\ref{deg8.sec}, see
Theorem~\ref{deg8.thm}. They are isotropic, and their anisotropic
kernel contains $\Qc$. Finally, we use Laurent power series in
\S\ref{large.sec} to construct algebras with involution of large
degree that do not contain $\Qc$.

\section{Notations and preliminary observations} 

We work over a base field $F$ of characteristic different from $2$,
and only consider algebras with involution of the first kind. 
We refer the reader to~\cite{KMRT} and to \cite{Lam} for background
information on central simple algebras with involution and on
quadratic forms. However, we depart from the notation in \cite{Lam} by
using $\pform{a_1,\dots, a_n}$ to denote the $n$-fold Pfister form
$\otimes_{i=1}^n\qform{1,-a_i}$. 

If $h\colon V\times V\to D$ is a regular hermitian or
skew-hermitian form on a finite-dimensional vector space $V$ over a
division algebra $D$, we denote by $\ad_h$ the involution on
$\End_D(V)$ that is adjoint to $h$. In the particular case where $D$
is split and $h$ is the polar form of a quadratic form $q$, we also
denote $\ad_q$ for $\ad_h$. Following Becher \cite{Becher}, for any
$n$-dimensional 
quadratic form $q$ over $F$, we denote by $\Ad_q$ the split orthogonal
algebra with involution $(M_n(F),\ad_q)$.

Recall that a central simple $F$-algebra with involution $(A,\sigma)$
is hyperbolic if and only if it is isomorphic to $(\End_DV,\ad_h)$ for
some hyperbolic hermitian or skew-hermitian form $h$ on a vector space $V$
over a division algebra $D$. In particular, $A$ admits a hyperbolic
involution if and only if the index of $A$ divides $\frac12\deg(A)$. 
If so, then $A$ admits up to conjugation a unique orthogonal hyperbolic
involution, and a unique symplectic one; see~\cite{BST} or~\cite[6.B]{KMRT}.
 
For any central simple $F$-algebra with involution $(A,\sigma)$ and
any field extension $K/F$, we let
$(A,\sigma)_K=(A\otimes_FK,\sigma\otimes\id)$. As pointed out in the
introduction, we shall be mostly interested in the special case where
$K=F_Q$ is the function field of the conic associated to a quaternion
algebra $Q$, i.e.\ its Severi--Brauer variety. Throughout the paper,
we fix the notation
\[
Q=(a,b)_F
\]
 with $a$, $b\in F^\times$, and we assume $Q$ is not split. We may then
identify $F_Q$ with the quadratic extension $F(\sqrt{at^2+b})$ of
$F(t)$, where $t$ is an indeterminate. For any central simple
$F$-algebra $A$, we let $A(t)=A\otimes_FF(t)$ and $A[t]=A\otimes_FF[t]$.

\begin{prop}
  Let $(A,\sigma)$ be a central simple $F$-algebra with
  involution. Assume $\sigma$ is anisotropic. The following conditions
  are equivalent:
  \begin{enumerate}
  \item[(a)]
  $(A,\sigma)_{F_Q}$ is hyperbolic;
  \item[(b)]
  $A(t)$ contains an element $y$ satisfying
  \begin{equation}\label{y.eq} 
    \sigma(y)=-y\qquad\text{and}\qquad y^2=at^2+b. 
  \end{equation}
  \end{enumerate}
  If conditions \emph{(a)} and \emph{(b)} hold, then $A[t]$
  contains an element 
  $y_0$ such that
  \begin{equation}
  \label{y0.eq}
  \sigma(y_0)y_0=-(at^2+b).
  \end{equation}
  Moreover, the following conditions are equivalent:
  \begin{enumerate}
  \item[(a')]
  $(A,\sigma)$ contains $\Qc$;
  \item[(b')]
  $A[t]$ contains an element $y$ satisfying \eqref{y.eq}.
  \end{enumerate}
\end{prop}

\begin{proof}
  The equivalence of (a) and (b) readily follows from the description
  of aniso\-tropic involutions that become hyperbolic over a quadratic
  extension in \cite[3.3]{BST}. If $y\in A(t)$ satisfies~\eqref{y.eq},
  then the version of the Cassels--Pfister
  theorem for algebras with involution in~\cite{T:96} yields an element
  $u\in A(t)$ such that $\sigma(u)u=1$ and $uy\in A[t]$. Then $y_0=uy$
  satisfies~\eqref{y0.eq}. 

  If (a') holds, then $A$ contains two
  skew-symmetric elements $i$, $j$ such that $i^2=a$, $j^2=b$, and
  $ji=-ij$. Then $y=it+j\in A[t]$ satisfies \eqref{y.eq}, so (b')
  holds. Conversely, suppose (b') holds and let $y\in A[t]$ satisfy
  \eqref{y.eq}. We then have $\sigma(y)y=-(at^2+b)$, hence the degree of
  $y$ is $1$ since $\sigma$ is anisotropic. Thus, $y=\lambda t+\mu$
  for some $\lambda$, $\mu\in A$. It follows from \eqref{y.eq} that
  $\lambda$ and $\mu$ are skew-symmetric and satisfy $\lambda^2=a$,
  $\mu^2=b$, and $\mu\lambda=-\lambda\mu$, hence they generate a
  $\sigma$-stable subalgebra of $(A,\sigma)$ isomorphic to $\Qc$.
\end{proof}

\begin{prop}
  \label{division.prop}
  Let $(A,\sigma)$ be a central simple $F$-algebra with
  involution. Assume $A$ is division. If $(A,\sigma)_{F_Q}$ is
  hyperbolic, then $A$ contains $Q$ and $A[t]$ contains an element
  $y_1$ such that $y_1^2=at^2+b$.
\end{prop}

\begin{proof}
  Since $(A,\sigma)_{F_Q}$ is hyperbolic, the algebra $A_{F_Q}$ is not
  division, hence Merkurjev's index reduction theorem \cite[Th.~1]{Mer:92}
  shows that $A$ contains $Q$, hence also two elements $i$, $j$ such
  that $i^2=a$, $j^2=b$, and $ji=-ij$. Then $y_1=it+b$ satisfies
  $y_1^2=at^2+b$. 
\end{proof}

Thus, the algebra with involution $(A,\sigma)$ in
Theorem~\ref{totdec.thm} below is such 
that $A(t)$ contains an element $y$ satisfying \eqref{y.eq}, and
$A[t]$ contains elements $y_0$, $y_1$ satisfying 
\[
\sigma(y_0)y_0=-(at^2+b),\qquad y_1^2=at^2+b,
\]
but no element satisfying both equations.
\medbreak
\par
For the rest of this section, we focus on the case where $\sigma$ is
orthogonal. We then have a discriminant $\disc\sigma\in
F^{\times}/F^{\times2}$ and a Clifford algebra $C(A,\sigma)$, see
\cite[\S\S7, 8]{KMRT}. Recall from \cite[(8.25)]{KMRT} that the center
of $C(A,\sigma)$ 
is the quadratic \'etale $F$-algebra obtained by adjoining a square root of 
$\disc\sigma$. Therefore, if $\disc\sigma=1$ the algebra $C(A,\sigma)$
decomposes into a direct product of two components, which are central
simple $F$-algebras,
\[
C(A,\sigma)\simeq C_+(A,\sigma)\times C_-(A,\sigma).
\]
If $\deg A\equiv0\bmod4$, then the canonical involution
$\underline{\sigma}$ on $C(A,\sigma)$ restricts to involutions
$\sigma_+$ and $\sigma_-$ on $C_+(A,\sigma)$ and
$C_-(A,\sigma)$. Moreover, the tensor product
$C_+(A,\sigma)\otimes_FC_-(A,\sigma)$ is Brauer-equivalent to $A$, see
\cite[(8.12), (9.14)]{KMRT}.

\begin{prop}
  \label{orthog.prop}
  Let $(A,\sigma)$ be a central simple $F$-algebra with orthogonal
  involution and $\deg A\equiv0\bmod4$. If $(A,\sigma)_{F_Q}$ is
  hyperbolic, then $\disc\sigma=1$ and at least one of
  $(C_+(A,\sigma),\sigma_+)$, $(C_-(A,\sigma),\sigma_-)$ is split and
  isotropic. 
\end{prop}

\begin{proof}
  Since $\sigma_{F_Q}$ is hyperbolic, $\disc\sigma$ is a square in
  $F_Q$, hence also in $F$ since $F$ is quadratically closed in
  $F_Q$. By \cite[(8.31)]{KMRT}, the Clifford algebra of any hyperbolic
  involution has a split component; therefore
  $C_+(A,\sigma)_{F_Q}$ or $C_-(A,\sigma)$ is split. The corresponding
  involution $\sigma_\pm$ is isotropic 
  over $F_Q$ by the main theorem in \cite{G:clif}.
\end{proof}

Since the Brauer group kernel of the scalar extension map from $F$ to
$F_Q$ is $\{0,[Q]\}$, the description of algebras $(A,\sigma)$ with
$\sigma$ orthogonal and $\deg A\equiv0\bmod4$ such that
$(A,\sigma)_{F_Q}$ is hyperbolic falls into two cases:
\begin{description}
\item[Case 1] $\bigl\{[C_+(A,\sigma)],[C_-(A,\sigma)]\bigr\}=
\bigl\{0,[A]\bigr\}$;
\item[Case 2]
  $\bigl\{[C_+(A,\sigma)],[C_-(A,\sigma)]\bigr\}=
\bigl\{[Q],[A\otimes_FQ]\bigr\}$.  
\end{description}
Following Garibaldi's definition in \cite{G:16}, case~1 is when
$(A,\sigma)\in I^3$.
Note that these two cases are not exclusive: they intersect if and
only if $A$ is Brauer-equivalent to $Q$; then
$(A,\sigma)$ becomes split hyperbolic over $F_Q$. This case is easily
dealt with in \S\ref{split.sec}. If $\deg A=4$, the first case arises
if and only if $(A,\sigma)$ is hyperbolic, and the second case if and
only if $(A,\sigma)$ contains $\Qc$, see \S\ref{deg4orth.sec}. For
algebras of degree~$8$, case~1 is when $(A,\sigma)$ is totally
decomposable, see~\S\ref{totdec.sec}; case~2 is considered
in~\S\ref{deg8.sec}. 

\section{Algebras with involution that become split hyperbolic over
  $F_Q$}
\label{split.sec}

Since $F_Q$ is the function field of the Severi--Brauer variety of
$Q$, the Brauer group kernel of the scalar extension from $F$ to $F_Q$
is $\{0,[Q]\}$. Therefore, if $(A,\sigma)$ is a central simple
$F$-algebra with involution that becomes split hyperbolic over $F_Q$,
then either $A$ is split or $A$ is Brauer-equivalent to $Q$. We
consider each case separately.

\begin{prop}
  \label{split.prop}
  Let $(A,\sigma)$ be a split central simple $F$-algebra with
  involution.
  \begin{enumerate}
  \item[(1)]
  If $\sigma$ is symplectic, then it is hyperbolic. In this case,
  $(A,\sigma)$ contains $\Qc$ if and only if $\deg A\equiv0\bmod4$.
  \item[(2)]
  If $\sigma$ is orthogonal, then $(A,\sigma)\simeq\Ad_q$ for some
  quadratic form $q$. We have $(\Ad_q)_{F_Q}$ hyperbolic if and only
  if the anisotropic kernel of $q$ is a multiple of the norm form
  $n_Q$. Assuming this condition holds, $\Ad_q$ contains $\Qc$ if and
  only if the Witt index of $q$ is a multiple of $4$.
  \end{enumerate}
\end{prop}

\begin{proof}
  (1) Since every alternating form on an $F$-vector space is
  hyperbolic, it follows that every split algebra with hyperbolic
  involution is hyperbolic. If $\rho$ is an orthogonal involution on
  $Q$, then $\Qc\otimes(Q,\rho)$ is a split algebra of degree~$4$ with
  symplectic involution. If $\deg A=4m$, we have
  \[
  (A,\sigma)\simeq\Qc\otimes(Q,\rho)\otimes\Ad_q
  \]
  for any quadratic form $q$ of dimension $m$ since all the hyperbolic
  involutions on $A$ are conjugate. Conversely, if $A$ contains $Q$,
  then the centralizer of $Q$ in $A$ is Brauer-equivalent to $Q$,
  hence of even degree. Therefore, $\deg A\equiv0\bmod4$.

  (2) By definition, $(\Ad_q)_{F_Q}$ is hyperbolic if and only if
  $q_{F_Q}$ is hyperbolic, hence the first statement follows from
  \cite[X(4.11)]{Lam}. Let $n_Q\otimes q_0$ be the anisotropic kernel
  of $q$. If the Witt index of $q$ is $4m$, then denoting by $\H$ the
  hyperbolic plane over $F$ we have
  \[
  q\simeq n_Q\otimes(q_0\perp m\H).
  \]
  Since $\Ad_{n_Q}\simeq\Qc\otimes\Qc$ (see \cite[(11.1)]{KMRT}) it
  follows that 
  \[
  \Ad_q\simeq\Qc\otimes\Qc\otimes\Ad_{q_0\perp m\H},
  \]
  so $\Ad_q$ contains $\Qc$. Conversely, assume $\Ad_q$ contains $\Qc$
  and let $(A_1,\sigma_1)$ be the centralizer of $\Qc$ in
  $(A,\sigma)$. Then
  \[
  \Ad_q\simeq\Qc\otimes(A_1,\sigma_1)
  \]
  hence $\sigma_1$ is symplectic and $A_1$ is Brauer-equivalent to
  $Q$. Therefore, there is a hermitian $Q$-form $h_1$ such that
  $\sigma_1\simeq\ad_{h_1}$. The anisotropic kernel of $h_1$ has a
  diagonalization $\qform{a_1,\ldots,a_r}_Q$ with $a_1$, \ldots,
  $a_r\in F^\times$; letting $m$ be the Witt index of $h_1$, we have
  \[
  h_1\simeq\qform{a_1,\ldots,a_r}_Q\perp m\H_Q.
  \]
  Consider the following quadratic form over $F$:
  \[
  q_1=\qform{a_1,\ldots,a_r}\perp m\H.
  \]
  Then $(A_1,\sigma_1)\simeq\Qc\otimes\Ad_{q_1}$, hence
  $\Ad_q\simeq\Ad_{n_Q\otimes q_1}$, and therefore $q$ is
  isometric to $n_Q\otimes q_1$ up to a scalar factor. Now,
  \[
  n_Q\otimes q_1=n_Q\otimes\qform{a_1,\ldots,a_r}\perp m(n_Q\otimes\H)
  = n_Q\otimes\qform{a_1,\ldots,a_r}\perp4m\H,
  \]
  and $n_Q\otimes\qform{a_1,\ldots, a_r}$ is anisotropic since
  $\qform{a_1,\ldots, a_r}_Q$ is anisotropic (see \cite[Ch.~10,
  Th.~1.1]{Scha}). Therefore, the Witt index of $q$ is $4m$.
\end{proof}

For algebras that are Brauer-equivalent to $Q$, the result is as
follows:

\begin{prop}
  Let $(A,\sigma)$ be a central simple $F$-algebra with involution
  such that $A$ is Brauer-equivalent to $Q$.
  \begin{enumerate}
  \item[(1)]
  If $\sigma$ is symplectic, then $(A,\sigma)$ contains $\Qc$ and
  $(A,\sigma)_{F_Q}$ is hyperbolic.
  \item[(2)]
  If $\sigma$ is orthogonal and $(A,\sigma)_{F_Q}$ is hyperbolic, then
  $(A,\sigma)$ is hyperbolic and contains $\Qc$.
  \end{enumerate}
\end{prop}

\begin{proof}
  (1) It was already observed in the proof of
  Proposition~\ref{split.prop}(2) that every central simple algebra
  with symplectic involution that is Brauer-equivalent to $Q$ contains
  $\Qc$. Therefore, every such algebra becomes hyperbolic over $F_Q$.

  (2) The first statement follows from the injectivity of the scalar
  extension map $W^-\Qc\to W(F_Q)$ proved independently in \cite{Dej:01}
  and \cite{PSS}. If $(A,\sigma)$ is hyperbolic, then $A=M_r(F)\otimes
  Q$ for some even integer $r$, and $\sigma$ is conjugate to
  $\tau\otimes\binv$ for any symplectic (hyperbolic) involution $\tau$
  on $M_r(F)$. In particular, $(A,\sigma)$ contains $\Qc$.
\end{proof}

Focusing on the anisotropic case in the propositions above, we have:

\begin{cor}
  \label{split.cor}
  Let $(A,\sigma)$ be a central simple $F$-algebra with
  involution. Assume $\sigma$ is anisotropic and $(A,\sigma)_{F_Q}$ is
  split hyperbolic. Then either
  \begin{itemize}
  \item 
  $A$ is split, $\sigma$ is orthogonal, and there is a quadratic form
  $q$ such that 
  \[
  (A,\sigma)\simeq\Ad_{n_Q\otimes
    q}\simeq\Qc\otimes\Qc\otimes\Ad_q;
  \]
  or
  \item
  $A$ is Brauer-equivalent to $Q$, $\sigma$ is symplectic, and there
  is a quadratic form $q$ such that 
  \[
  (A,\sigma)\simeq\Qc\otimes\Ad_q.
  \]
  \end{itemize}
  In both cases, $(A,\sigma)$ contains $\Qc$.
\end{cor}

Using the results in this section, we may describe the algebras with
involution of degree $2\bmod4$ that become hyperbolic over $F_Q$:

\begin{cor}
  \label{split2.cor}
  Let $(A,\sigma)$ be a central simple $F$-algebra with
  involution. If $\deg A\equiv2\bmod4$ and $(A,\sigma)_{F_Q}$ is
  hyperbolic, then $A_{F_Q}$ is split. If moreover $\sigma$ is
  anisotropic, then it is symplectic and there is an
  odd-dimensional quadratic form $q$ over $F$ such that
  $(A,\sigma)\simeq\Qc\otimes\Ad_q$.
\end{cor}

\begin{proof}
  Since $\deg A\equiv2\bmod4$, we have $A\simeq M_r(H)$ for some odd
  integer $r$ and some quaternion algebra $H$ over $F$. If $H$ is
  division, $A$ does not admit a hyperbolic involution. Therefore, the
  hypothesis that $(A,\sigma)_{F_Q}$ is hyperbolic implies that
  $H_{F_Q}$ is split, hence $(A,\sigma)_{F_Q}$ is split
  hyperbolic. The last statement then readily follows from
  Corollary~\ref{split.cor}. 
\end{proof}

In view of this corollary, we only consider central simple
algebras of degree divisible by~$4$ in the following sections.

\section{Algebras of degree 4}
\label{deg4.sec}

Throughout this section, $A$ is a central simple algebra of
degree~$4$ over an arbitrary field $F$ of characteristic different
from~$2$. Involutions on $A$ are classified by cohomological 
invariants (see \cite[Ch.~4]{KMRT}), which we use to give an explicit
description of the involutions on $A$ that become hyperbolic over
$F_Q$. 

\subsection{The orthogonal case}
\label{deg4orth.sec}
This case is easy to handle using Clifford algebras. 

\begin{prop}
  \label{deg4orth.prop}
  Let $\sigma$ be an orthogonal involution on $A$.
  \begin{enumerate}
  \item[(1)]
  $(A,\sigma)$ is hyperbolic if and only if $\disc\sigma=1$ and one of
  the components of $C(A,\sigma)$ is split;
  \item[(2)]
  $(A,\sigma)$ contains $\Qc$ if and only if $\disc\sigma=1$ and one
  of the components of $C(A,\sigma)$ is isomorphic to $Q$;
  \item[(3)]
  $A$ contains $Q$ if and only if $A\otimes_FQ$ is Brauer-equivalent
  to a quaternion algebra.
  \end{enumerate}
\end{prop}

\begin{proof}
  (1) This readily follows from \cite[2.5]{BST}.

  (2) If $(A,\sigma)$ contains $\Qc$, then there is a quaternion
  $F$-algebra $Q'$ with canonical involution $\binv$ such that
  \[
  (A,\sigma)\simeq\Qc\otimes(Q',\binv).
  \]
  By \cite[(15.12)]{KMRT}, this relation holds if and only if
  $C(A,\sigma)\simeq Q\times Q'$. This proves~(2).

  (3) If $A$ contains (a copy of) $Q$, then the centralizer of $Q$ in
  $A$ is a quaternion $F$-algebra $Q'$, and we have
  \[
  A\simeq Q\otimes_FQ'.
  \]
  This relation holds if and only if $A\otimes_FQ$ is
  Brauer-equivalent to $Q'$, hence (3) follows.
\end{proof}

\begin{cor}
  Let $\sigma$ be an orthogonal involution on $A$. Then
  $(A,\sigma)_{F_Q}$ is hyperbolic if and only if $(A,\sigma)$ is
  hyperbolic or contains $\Qc$.
\end{cor}

\begin{proof}
  As noticed in the introduction, any $(A,\sigma)$ containing $\Qc$ is
  hyperbolic over $F_Q$, so we only have to prove the converse. Assume
  $(A,\sigma)_{F_Q}$ hyperbolic. Since $F$ is quadratically closed in
  $F_Q$, the discriminant of $\sigma$ is trivial, so $C(A,\sigma)$ is
  a direct product of two quaternion algebras. At least one of these
  quaternion algebras splits over $F_Q$ by
  Proposition~\ref{deg4orth.prop}(1), hence that component must be
  either split or isomorphic to $Q$. It follows that $(A,\sigma)$ is
  either hyperbolic or contains $\Qc$, by
  Proposition~\ref{deg4orth.prop}. 
\end{proof}

Note that a central simple algebra of degree~$4$ with hyperbolic
involution does not necessarily contain $\Qc$, even if it contains
$Q$: if $Q'$ is a quaternion $F$-algebra such that $Q\otimes_FQ'$ has
index~$2$, then $(A,\sigma)=(M_2(F),\binv)\otimes(Q',\binv)$ is
hyperbolic over $F$ and satisfies $C(A,\sigma)\simeq M_2(F)\times Q'$
by \cite[(15.12)]{KMRT}. Therefore, Proposition~\ref{deg4orth.prop}(2) shows
that $(A,\sigma)$ does not contain $\Qc$, even though
Proposition~\ref{deg4orth.prop}(3) shows that $A$ contains
$Q$. This situation does not occur in the symplectic case, in view of
Proposition~\ref{deg4symp.prop} below. 

\subsection{The symplectic case}
Symplectic involutions on $A$ are classified up to conjugation by a
relative invariant $\Delta$ with values in the Galois cohomology group
$H^3(F,\mu_2)$, see\footnote{%
The discussion in \cite{KMRT} is in terms of the $3$-fold Pfister form
$j$ whose Arason invariant is $\Delta$.}
\cite[(16.9)]{KMRT}. We denote by $[A]$ the Brauer class of $A$,
viewed as an element of $H^2(F,\mu_2)$, and for $\lambda\in F^\times$
we denote by $(\lambda)$ the square class of $\lambda$, viewed as an
element in $H^1(F,\mu_2)$. Using the invariant $\Delta$, we show:

\begin{prop}
  \label{deg4symp.prop}
  Let $\sigma$ be a symplectic involution on $A$. If $A$ contains $Q$
  and $(A,\sigma)_{F_Q}$ is hyperbolic, then $(A,\sigma)$ contains $\Qc$.
\end{prop}

\begin{proof}
  Since $A$ contains $Q$, it decomposes as $A=Q\otimes Q'$ for some
  quaternion algebra $Q'=(a',b')_F$. Let $\tau'$ be an orthogonal
  involution on $Q'$ of discriminant $a'$. The involution
  $\tau=\binv\otimes\tau'$ on $A=Q\otimes Q'$ is of symplectic type,
  and clearly hyperbolic over $F_Q$ since $(A,\tau)$ contains
  $\Qc$. Similarly, for every invertible $y\in\Sym(Q',\tau')$, the
  involution 
  \[
  \tau_y=\Int(1\otimes y)\circ\tau=\binv\otimes(\Int(y)\circ\tau)
  \]
  is such that $(A,\tau_y)$ contains $\Qc$. We prove below that
  $\sigma$ is conjugate to $\tau_y$ for a suitable $y$. By
  \cite[(16.18)]{KMRT}, the relative discriminant $\Delta_\tau(\tau')$
  is given by
  \[
  \Delta_\tau(\tau')=(\Nrp_\tau(1\otimes y))\cup [A]\in H^3(F,\mu_2),
  \]
  where $\Nrp_\tau$ is the Pfaffian norm, as defined
  in~\cite[(2.9)]{KMRT}. Since this relative discriminant classifies
  symplectic involutions on $A$ up to conjugation, it suffices to show: 
\renewcommand{\qed}{\relax}
\end{proof}

\begin{lem}
  \label{deg4symp.lem}
  There exists an invertible element $y\in \Sym(Q',\tau')$ such that 
  \[
  \Delta_\tau(\sigma)=(\Nrp_\tau(1\otimes y))\cup [A].
  \]
\end{lem}

\begin{proof}
  The involution $\sigma$ can be written as $\sigma=\Int(x)\circ \tau$
  for some $x\in\Sym(A,\sigma)$, hence we already have  
  \begin{equation} 
  \label{disc.eq} 
  \Delta_\tau(\sigma)=(\Nrp_\tau(x))\cup [A],  
  \end{equation} 
  and we want to prove we can substitute for $x$ some element
  $1\otimes y$ with $y\in\Sym(Q',\tau')$. Since both $\sigma$ and
  $\tau$ become hyperbolic over $F_Q$, the relative discriminant
  $\Delta_\tau(\sigma)$ is killed by $F_Q$, hence  
  $(\Nrp_\tau(x))\cup [Q']=(\lambda)\cup[Q]$ for some $\lambda\in
  F^\times$. By the common slot lemma~\cite[Lemma 1.7]{A}, we may
  even assume  
  \begin{equation} 
  (\Nrp_\tau(x))\cup [Q']=(\lambda)\cup[Q']=(\lambda)\cup[Q],
  \end{equation} 
  from which we deduce 
  \begin{equation}
  \label{la.eq}
  (\lambda)\cup([Q]+[Q'])=0,\quad\text{and}\quad(\lambda\Nrp_\tau(x))\cup
  [Q']=0.  
  \end{equation} 
  Hence the relative discriminant of $\sigma$ is 
  \begin{equation*}
  \Delta_\tau(\sigma)=(\Nrp_\tau(x))\cup([Q]+[Q'])=(\lambda
  \Nrp_\tau(x))\cup([Q]+[Q'])=(\lambda \Nrp_\tau(x))\cup[Q].
  \end{equation*}
  Moreover, the quadratic space $(\Sym(A,\sigma),\Nrp_\tau)$ is an
  Albert quadratic space for the biquaternion algebra $A$
  by~\cite[(16.8)]{KMRT}. Hence, its Clifford invariant is
  $e_2(\Nrp_\tau)=[A]=[Q]+[Q']$. We deduce from~(\ref{la.eq}) that the
  quadratic form $\pform\lambda\otimes\Nrp_\tau$ has trivial Arason
  invariant $e_3$, and hence is hyperbolic by the Arason--Pfister
  Hauptsatz. So there exists an element $x'\in\Sym(A,\sigma)$ such that  
  $\lambda\Nrp_\tau(x)=\Nrp_\tau(x')$, and we have proven 
  \begin{equation} \label{Q'.eq} 
  \Delta_\tau(\sigma)=(\Nrp_\tau(x'))\cup [Q]
  \quad\text{and}\quad(\Nrp_\tau(x'))\cup [Q']=0 
  \end{equation} 
  for some $x'\in\Sym(A,\sigma)$. 

  Let us now pick a pure quaternion $i'\in Q'$ such that $i'^2=a'$ and
  $\tau'=\Int(i')\circ\binv$. Denoting by $Q^0$ the vector space of
  pure quaternions in $Q$, we have
  \begin{equation*} 
  \Sym(A,\tau)=(1\otimes\Sym(Q',\tau'))\oplus(Q^0\otimes i'). 
  \end{equation*} 
  Hence $x'$ can be written as $x'=x_0+1\otimes\xi_1+\xi_2\otimes i'$
  for some $x_0\in F$, some pure quaternion $\xi_1\in
  Q'^0\cap\Sym(Q',\tau')$ and some $\xi_2\in Q^0$. The pure quaternion
  $\xi_1$ being $\tau'$-symmetric, it anticommutes with $i'$, and
  changing of quaternionic basis if necessary, we may assume that
  $\xi_1^2=b'x_1^2$ for some $x_1\in F$. Similarly, we may assume that
  $\xi_2^2=ax_2^2$ for some $x_2\in F$. Note that we allow $x_i=0$ for
  $i=0$, $1$, and $2$ since some term might be zero in the
  decomposition of $x'$. We then have 
  \begin{equation} \label{x'.eq} 
  \Nrp_\tau(x')=x_0^2-b'x_1^2-aa'x_2^2,
  \end{equation}
  and the following lemma finishes the proof: 
\renewcommand{\qed}{\relax}
\end{proof}

\begin{lem}
  \label{fin.lem} 
  There exists $z_0$, $z_1$, $y_0$, $y_1$ and $y_2$ such that 
  $\Nrp_\tau(x')(z_0^2-az_1^2)=y_0^2-b'y_1^2+a'b'y_2^2\not=0$.
\end{lem}

Indeed, if we let $y=y_0+y_1j'+y_2i'j'$, where $j'$ is a pure
quaternion in $Q'$ such that $j'^2=b'$ and $i'j'=-j'i'$, we have 
$\Nrp_\tau(1\otimes y)=Nrd_{Q'}(y)=y_0^2-b'y_1^2+a'b'y_2^2$
(see~\cite[(2.11)]{KMRT}). Hence, 
\begin{align*}
\Delta_\tau(\sigma)& =(\Nrp_\tau(x'))\cup
[(a,b)_F]=(\Nrp_\tau(x')(z_0^2-az_1^2))\cup
[(a,b)_F]\\
& =(\Nrd_{Q'}(y))\cup[A]=(\Nrp_\tau(1\otimes y))\cup[A],
\quad\text{ with }y\in\Sym(Q',\tau').
\end{align*}

\begin{proof}[Proof of Lemma~\ref{fin.lem}]  
  If $\qform{1,-b',a'b'}$ is isotropic, then it is universal, hence
  represents $\Nrp_\tau(x')$ and we are done. Otherwise, denote
  $\mu=\Nrp_\tau(x')$; by~(\ref{x'.eq}), we have
  $(\mu)\cup[(aa',b')_F]=0$, and combining 
  with~\eqref{Q'.eq}, we get $(\mu)\cup[(a,b')]=0$. In terms of
  quadratic forms, this means that the $3$-fold Pfister form  
  \begin{equation} \label{pform.eq}
  \pform{\mu,a,b'}=\qform{1,-\mu,-a,\mu b'}\perp\qform{\mu
  a,ab',-\mu ab'}\perp\qform{-b'}
  \end{equation} 
  is hyperbolic. On the other hand, by~\eqref{x'.eq}, the quadratic
  form $\qform{1,-b',-aa',-\mu}$ is isotropic. Multiplying by
  $(-ab'\mu)$, we get that $-a'b'\mu$ is represented by the form
  $\qform{\mu a, ab',-\mu ab'}$. Hence, in view of~\eqref{pform.eq},
  the quadratic form $\qform{1,-\mu,-a,\mu b',-a'b'\mu}$, which is a
  $5$-dimensional subform of $\pform{\mu,a,b'}$, is necessarily
  isotropic. Since the quadratic forms $\qform{1,-a}$ and
  $\qform{1,-b',a'b'}$ are anisotropic, this 
  completes the proof. 
\end{proof} 

Using Proposition~\ref{deg4symp.prop}, we may characterize the
symplectic involutions on $A$ that become hyperbolic over $F_Q$:

\begin{thm}
  \label{deg4symp.thm}
  Let $\sigma$ be a symplectic involution on $A$. If
  $(A,\sigma)_{F_Q}$ is hyperbolic, then either
  \begin{enumerate}
  \item[(a)] 
  $(A,\sigma)\simeq\Qc\otimes(Q',\rho)$ for some quaternion
  $F$-algebra with orthogonal involution $(Q',\rho)$, or
  \item[(b)]
  $(A,\sigma)\simeq\Ad_{\pform{\lambda}}\otimes(Q',\binv)$ for some
  quaternion $F$-algebra $Q'$ and some $\lambda\in F^\times$ with the
  following properties:
  $Q\otimes_FQ'$ is a division algebra and the norm forms $n_Q$, $n_{Q'}$
  satisfy $\pform{\lambda}\cdot n_Q\simeq\pform{\lambda}\cdot n_{Q'}$.
  \end{enumerate}
  Conversely, if $(A,\sigma)$ is of either type above, then
  $(A,\sigma)_{F_Q}$ is hyperbolic.
\end{thm}

Note that (a) and (b) are mutually exclusive since in the first
case $A\otimes_FQ$ has Schur index~$1$ or $2$, whereas it has
index~$4$ in case~(b).

\begin{proof}
  Suppose $(A,\sigma)$ is not as in case~(a) and $(A,\sigma)_{F_Q}$ is
  hyperbolic. Then
  Proposition~\ref{deg4symp.prop} shows that $A$ does not contain
  $Q$. In particular, $A$ is not division, by
  Proposition~\ref{division.prop}, so $A\simeq M_2(Q')$ for some
  quaternion $F$-algebra $Q'$, and $Q\otimes_FQ'$ is a division algebra
  by Proposition~\ref{deg4orth.prop}(3). As observed in the proof of
  Proposition~\ref{split.prop}(2), $(A,\sigma)$ contains $(Q',\binv)$,
  hence
  \[
  (A,\sigma)\simeq\Ad_{\pform{\lambda}}\otimes(Q',\binv)
  \qquad\text{for some $\lambda\in F^\times$}.
  \]
  The algebra $A$ carries a hyperbolic symplectic involution $\tau$,
  and we have by~\cite[(16.21)]{KMRT}
  \[
  \Delta_\tau(\sigma)=(\lambda)\cup [Q'].
  \]
  Since $(A,\sigma)_{F_Q}$ is hyperbolic, this invariant vanishes over
  $F_Q$. Hence, by~\cite[X(4.11)]{Lam} and Arason's common slot lemma
  we may assume as 
  in the proof of Lemma~\ref{deg4symp.lem} that $(\lambda)\cup
  [Q']=(\lambda)\cup [Q]$. Therefore, $\pform{\lambda}\cdot
  n_{Q'}\simeq\pform{\lambda}\cdot n_Q$, which shows that $(A,\sigma)$
  is as in case~(b).

  It is clear that $(A,\sigma)_{F_Q}$ is hyperbolic in case~(a). In
  case (b) the algebra $A$ carries a hyperbolic symplectic involution
  $\tau$ and $\Delta_\tau(\sigma)=(\lambda)\cup[Q']$ vanishes over
  $F_Q$. Therefore, $(A,\sigma)_{F_Q}$ is hyperbolic.
\end{proof}

\section{$F_Q$-minimal quadratic forms of dimension 5}
\label{min.sec}

A quadratic form $\varphi$ over $F$ is called \emph{$F_Q$-minimal} if
$\varphi_{F_Q}$ is isotropic and $\psi_{F_Q}$ is anisotropic for every
proper subform $\psi\subset\varphi$. In this section, we show that
Theorem~\ref{deg4symp.thm} can be used to recover (and is in fact
equivalent to) the description of $F_Q$-minimal forms of dimension~$5$
due to Hoffmann, Lewis, and Van Geel \cite[Prop.~4.1]{HLVG}.

A general procedure to construct central simple algebras with
involution that become hyperbolic over $F_Q$ uses Clifford
algebras. Recall from \cite[V(1.9), V(2.4)]{Lam} that for any
quadratic form $\varphi$ of odd dimension $2m+1$ over $F$ the even Clifford
algebra $C_0(\varphi)$ is central simple over $F$ of degree $2^m$. It
carries a canonical involution $\tau_0$, which is the restriction of
the involution on the full Clifford algebra that leaves invariant
every vector in the underlying vector space of $\varphi$. The involution
$\tau_0$ is orthogonal if $m\equiv0$ or $3\bmod4$ and symplectic
otherwise, see \cite[(8.4)]{KMRT}.

\begin{prop}
  \label{qf.prop}
  Let $\varphi$ be a quadratic form of odd dimension over $F$.
  \begin{enumerate}
  \item[(1)]
  If $\varphi_{F_Q}$ is isotropic, then $(C_0(\varphi),\tau_0)_{F_Q}$ is
  hyperbolic. The converse holds if $\dim\varphi=5$.
  \item[(2)]
  If $\varphi$ contains a subform similar to $\qform{1,-a,-b}$, then
  $(C_0(\varphi),\tau_0)$ contains $\Qc$. The converse holds if
  $\dim\varphi=5$. 
  \end{enumerate}
\end{prop}

\begin{proof}
  (1) The first statement readily follows from \cite[(8.5)]{KMRT} and
  the second from \cite[(15.21)]{KMRT}. 

  (2) Suppose the underlying vector space of $\varphi$ contains
  orthogonal vectors 
  $e_0$, $e_1$, $e_2$ satisfying for some $\lambda\in F^\times$
  \[
  \varphi(e_0)=\lambda,\qquad \varphi(e_1)=-\lambda a,\qquad
  \varphi(e_2)=-\lambda b. 
  \]
  Then the products $e_0e_1$ and $e_0e_2$ generate a $\tau_0$-stable
  subalgebra of $C_0(\varphi)$ isomorphic to $\Qc$.

  For the rest of the proof, suppose $\dim\varphi=5$ and
  $(C_0(\varphi),\tau_0)$ contains $\Qc$. The centralizer of $\Qc$ is
  a quaternion algebra with orthogonal involution $(Q',\rho)$ such
  that
  \[
  (C_0(\varphi),\tau_0)=\Qc\otimes(Q',\rho).
  \]
  Let $V\subset\Sym(C_0(\varphi),\tau_0)$ be the vector space of
  $\tau_0$-symmetric elements of trace~$0$. The map $x\mapsto x^2$
  defines a quadratic form $s\colon V\to F$ that is similar to
  $\varphi$ by the equivalence $\mathsf{B}_2\equiv\mathsf{C}_2$, see
  \cite[(15.16)]{KMRT}. Let $y\in Q'$ be a $\rho$-skew-symmetric unit
  and let $Q^0\subset Q$ be the vector space of pure quaternions. The
  restriction of $s$ to the subspace $Q^0\otimes y\subset V$ is
  similar to $\qform{1,-a,-b}$, hence the proof is complete.
\end{proof}

\begin{cor}[Hoffmann--Lewis--Van Geel {\cite[Prop.~4.1]{HLVG}}]
  \label{qf.cor}
  A $5$-dimensional quadratic form $\varphi$ over $F$ is $F_Q$-minimal
  if and only if the following conditions hold:
  \begin{enumerate}
  \item[(a)]
  $\varphi$ is similar to a Pfister neighbour of the $3$-fold Pfister
  form $\pform{a,b,\lambda}$ for some $\lambda\in F^\times$, and
  \item[(b)]
  $C_0(\varphi)\simeq M_2(Q')$ for some quaternion $F$-algebra $Q'$
  such that $Q\otimes_FQ'$ is a division algebra.
  \end{enumerate}
\end{cor}

\begin{proof}
  Proposition~\ref{qf.prop} shows that $\varphi$ is $F_Q$-minimal if
  and only if $(C_0(\varphi),\tau_0)$ is hyperbolic but
  $(C_0(\varphi),\tau_0)$ does not contain $\Qc$. By
  Theorem~\ref{deg4symp.thm}, this condition is equivalent to
  \begin{equation}
    \label{eq:qf}
    (C_0(\varphi),\tau_0)\simeq\Ad_{\pform{\lambda}}\otimes(Q',\binv)
  \end{equation}
  for some quaternion $F$-algebra $Q'=(a',b')_F$ and some $\lambda\in
  F^\times$ with $Q\otimes Q'$ a division algebra and
  $\pform{a,b,\lambda}\simeq\pform{a',b',\lambda}$. It follows from
  the isomorphism~\eqref{eq:qf} that $\varphi$ is similar to a Pfister
  neighbour of $\pform{a',b',\lambda}$, by \cite[p.~271]{KMRT}. Thus,
  (a) and (b) hold if $\varphi$ is $F_Q$-minimal. Conversely, if
  $C_0(\varphi)\simeq M_2(Q')$ for some quaternion $F$-algebra
  $Q'=(a',b')_F$, then as observed in the proof of
  Proposition~\ref{split.prop}(2) we have
  \[
  (C_0(\varphi),\tau_0)\simeq\Ad_{\pform{\mu}}\otimes(Q',\binv)
  \qquad\text{for some $\mu\in F^\times$}.
  \]
  It then follows from \cite[p.~271]{KMRT} that $\varphi$ is similar
  to a Pfister neighbour of $\pform{a',b',\mu}$. If~(a) holds, then
  $\pform{a,b,\lambda}\simeq\pform{a',b',\mu}$ and by the common slot
  lemma we may assume $\lambda=\mu$. Thus, $\varphi$ is $F_Q$-minimal
  if~(a) and (b) hold.
\end{proof}

\section{Totally decomposable orthogonal involutions of degree 8}
\label{totdec.sec}

In this section, $A$ denotes a central simple $F$-algebra of
degree~$8$ and $\sigma$ is an orthogonal involution on $A$. The
algebra with involution $(A,\sigma)$ is called \emph{totally
  decomposable} if there are $\sigma$-stable quaternion subalgebras
$Q_1$, $Q_2$, $Q_3$ in $A$ such that $A=Q_1\otimes Q_2\otimes
Q_3$. Denoting by $\sigma_i$ the restriction of $\sigma$ to $Q_i$ for
$i=1$, $2$, $3$, we then have
\[
(A,\sigma)=(Q_1,\sigma_1)\otimes(Q_2,\sigma_2)\otimes(Q_3,\sigma_3).
\]
The totally decomposable algebras of degree~$8$ are characterized by
the property that $\disc\sigma=1$ and one of the components of the
Clifford algebra $C(A,\sigma)$ is split, see \cite[(42.11)]{KMRT}. 
Proposition~\ref{totdec.prop} below shows how to use this criterion to
relate totally decomposable algebras to quadratic forms. 

Recall that for any quadratic form $\varphi$ of dimension~$8$ with
$\disc\varphi=1$ the even Clifford algebra decomposes into a direct
product of central simple $F$-algebras of degree~$8$,
\[
C_0(\varphi)\simeq C_+(\varphi)\times C_-(\varphi).
\]
The canonical involution $\tau_0$ on $C_0(\varphi)$ restricts to
orthogonal involutions $\tau_+$, $\tau_-$ on $C_+(\varphi)$ and
$C_-(\varphi)$, and we have
\[
(C_+(\varphi),\tau_+)\simeq(C_-(\varphi),\tau_-).
\]
It is easily checked that $(C_+(\varphi),\tau_+)$ is totally
decomposable.

\begin{prop}
  \label{totdec.prop}
  For every central simple algebra of degree~$8$ with totally
  decomposable orthogonal involution $(A,\sigma)$, there is a
  quadratic form $\varphi$ with $\dim\varphi=8$ and $\disc\sigma=1$
  such that
  \[
  (A,\sigma)\simeq(C_+(\varphi),\tau_+).
  \]
  The form $\varphi$ is uniquely determined by $(A,\sigma)$ up to
  similarity. Moreover,
  \begin{itemize}
  \item[--]
  the algebra $A$ is split if and only if $\varphi$ is a
  multiple of a $3$-fold Pfister form; in that case
  $(A,\sigma)\simeq\Ad_\varphi$;
  \item[--]
  the algebra $(A,\sigma)$ contains
  $\Qc$ if and only if $\varphi$ contains a subform similar to
  $\qform{1,-a,-b}$.
  \end{itemize}
  Furthermore, the following conditions are
  equivalent:
  \begin{enumerate}
  \item[(a)]
  $(A,\sigma)$ is isotropic;
  \item[(b)]
  $(A,\sigma)$ is hyperbolic;
  \item[(c)]
  $\varphi$ is isotropic.
  \end{enumerate}
\end{prop}

\begin{proof}
  Since $(A,\sigma)$ is totally decomposable, it follows from
  \cite[(42.11)]{KMRT} that one of the components $C_+(A,\sigma)$ of the
  Clifford algebra is split. The canonical involution $\sigma_+$ on
  $C_+(A,\sigma)$ is orthogonal, hence there is a quadratic form
  $\varphi$ of dimension~$8$ such that
  \[
  (C_+(A,\sigma),\sigma_+)\simeq\Ad_\varphi.
  \]
  By triality (see \cite[(42.3)]{KMRT}) we have
  \[
  (C_0(\varphi),\tau_0)\simeq(A,\sigma)\times(A,\sigma).
  \]
  Therefore, $\disc\varphi=1$ and $(A,\sigma)\simeq
  (C_+(\varphi),\tau_+)$. Conversely, triality also shows that if
  $(A,\sigma)\simeq(C_+(\varphi),\tau_+)$, then the canonical
  involution $\underline{\sigma}$ on $C(A,\sigma)$ satisfies
  \[
  (C(A,\sigma),\underline{\sigma})\simeq\Ad_\varphi\times(A,\sigma),
  \]
  hence the form $\varphi$ is uniquely determined up to similarity.

  The Clifford algebra of $\varphi$ splits if and only if $\varphi$ is
  a multiple of a $3$-fold Pfister form, by~\cite[Ch.~2,
  Th.~14.4]{Scha} and \cite[X(5.6)]{Lam}. When that condition
  holds we have $(A,\sigma)\simeq\Ad_\varphi$ by \cite[(35.1)]{KMRT}.

  If $\varphi$ contains a multiple of $\qform{1,-a,-b}$, then the same
  argument as in the proof of Proposition~\ref{qf.prop}(2) shows that
  $(C_0(\varphi),\tau_0)$ contains $\Qc$. Projecting on each component, it
  follows that $(C_+(\varphi),\tau_+)$ contains $\Qc$. Conversely, if
  $(A,\sigma)$ contains $\Qc$, then we have
  \begin{equation}
    \label{eq:2}
    (A,\sigma)=\Qc\otimes(A_1,\sigma_1)
  \end{equation}
  for some central simple algebra with symplectic involution
  $(A_1,\sigma_1)$ of degree~$4$. By \cite[(15.19)]{KMRT} there is a
  $5$-dimensional quadratic form $\psi$ such that $\disc\psi=1$ and
  $(A_1,\sigma_1)\simeq(C_0(\psi),\tau_0)$. Letting $\tau'_0$ be the
  canonical involution on $C_0(\qform{ab,-a,-b})$, we may rewrite
  \eqref{eq:2} as
  \begin{equation}
    \label{eq:1}
    (A,\sigma)\simeq(C_0(\qform{ab,-a,-b}),\tau'_0)\otimes
    (C_0(\psi),\tau_0). 
  \end{equation}
  In view of the canonical embedding
  \[
  C_0(\qform{ab,-a,-b})\otimes_FC_0(\psi)\hookrightarrow
  C_0(\qform{ab,-a,-b}\perp\psi), 
  \]
  which is compatible with the canonical involutions, \eqref{eq:1}
  yields
  \[
  (A,\sigma)\simeq(C_+(\qform{ab,-a,-b}\perp\psi),\tau_+).
  \]
  Uniqueness of $\varphi$ shows that $\varphi$ is similar to
  $\qform{ab,-a,-b}\perp\psi$, hence it contains a subform similar to
  $\qform{1,-a,-b}$. 

  The equivalence of (a), (b), (c) is clear if $A$ is split, since
  then $(A,\sigma)\simeq\Ad_\varphi$ and $\varphi$ is a $3$-fold
  Pfister form, hence it is isotropic if and only if it is
  hyperbolic. For the rest of the proof, we may thus assume $A$ is not
  split. If $(A,\sigma)$ is hyperbolic, then it follows from
  \cite{G:clif} that the split component of
  $(C(A,\sigma),\underline{\sigma})$ is isotropic, hence
  (b)$\Rightarrow$(c). Conversely, if (c) holds, then
  \cite[(8.5)]{KMRT} shows that $(C_0(\varphi),\tau_0)$ is hyperbolic,
  hence $(C_+(\varphi),\tau_+)$ also is hyperbolic, proving
  (c)$\Rightarrow$(b). The equivalence of (a) and (b) readily follows
  from \cite[Prop.~2.10]{BFPQM}. 
\end{proof}

To give an example of a division $F$-algebra of degree~$8$ with a totally
decomposable orthogonal involution that is hyperbolic over $F_Q$ but
does not contain $\Qc$, we use an example of $F_Q$-minimal quadratic
form of dimension~$7$ due to Hoffmann and Van Geel \cite{HVG}. For the
rest of this section, we fix the following notation: $F_1=F_0(t,u)$ is
the function field in two independent indeterminates over an arbitrary
field $F_0$ of characteristic different from~$2$, and
$F=F_1((a))((b))$ is the iterated Laurent series field in two
indeterminates $a$, $b$. In accordance with our running notation, $Q$
denotes the quaternion algebra $(a,b)_F$. Let
\[
\varphi_0=\qform{1+t,u}\perp\qform{-a}\qform{1,u}
\perp\qform{-b}\qform{1,t+u}\perp\qform{ab}\qform{t}
\]
(see \cite[p.~43]{HVG}) and
\[
\varphi=\varphi_0\perp\qform{ab}\qform{t(1+t)(t+u)},
\]
so $\dim\varphi=8$ and $\disc\varphi=1$. Let also
\[
(A,\sigma)=(C_+(\varphi),\tau_+)\quad(=(C_0(\varphi_0),\tau_0)),
\]
a central simple $F$-algebra of degree~$8$ with a totally decomposable
involution.

\begin{thm}
  \label{totdec.thm}
  The algebra with involution $(A,\sigma)$ does not contain $\Qc$, yet
  $(A,\sigma)_{F_Q}$ is hyperbolic. Moreover, $A$ is a division algebra.
\end{thm}

\begin{proof}
  Since $\qform{1,t}\simeq\qform{1+t,t(1+t)}$ and
  $\qform{t,u}\simeq\qform{t+u,tu(t+u)}$, we have
  \begin{equation}
    \label{eq:3}
    \varphi_0\perp \qform{t(1+t),-at,-btu(t+u), ab, abu} \simeq
    \qform{1,t,u}\pform{a,b}. 
  \end{equation}
  Since the right side is hyperbolic over $F_Q$, it follows that
  $(\varphi_0)_{F_Q}$ is isotropic, and therefore $(A,\sigma)_{F_Q}$
  is hyperbolic by Proposition~\ref{qf.prop}(1) or \ref{totdec.prop}.

  To show $(A,\sigma)$ does not contain $\Qc$, we prove $\varphi$ does
  not contain any subform similar to $\qform{1,-a,-b}$. As
  in~\cite[p.~43]{HVG}, we decompose $\varphi$ as 
  \begin{equation}\label{decqf} 
  \varphi=\alpha\perp\qform{-a}\beta\perp\qform{-b}\gamma
  \perp\qform{ab}\delta, 
  \end{equation}
  where $\alpha=\qform{1+t,u}$, $\beta=\qform{1,u}$,
  $\gamma=\qform{1,t+u}$, and $\delta=\qform{t,t(1+t)(t+u)}$. By
  Springer's theorem, the isometry classes of $\alpha$, $\beta$,
  $\gamma$, and $\delta$ over $F_1$ are uniquely determined by
  $\varphi$. If three of those quadratic forms represent a common
  value $\lambda$, then $\varphi$ contains a $3$-dimensional subform
  of $\qform{\lambda}\pform{a,b}$, hence a subform similar to
  $\qform{1,-a,-b}$. We claim that the converse also holds. Indeed,
  assume first that $\varphi$ contains
  $\qform{\lambda}\qform{1,-a,-b}$ for some $\lambda\in
  F_1^\times$. Writing 
  $\varphi=\qform{\lambda}\qform{1,-a,-b}\perp\varphi'$ as
  in~\eqref{decqf}, we get, by uniqueness of the forms $\alpha$,
  $\beta$ and $\gamma$ up to isometry, that all three represent
  $\lambda$. Consider now the general situation, where $\lambda\in
  F^\times$ need not be in $F_1$; modifying it by a square if
  necessary, we may write it as $\lambda_0$, $-a\lambda_0$,
  $-b\lambda_0$ or $ab\lambda_0$ for some $\lambda_0\in
  F_1^\times$. The same argument as above then proves the claim.  
 
  Thus, to prove that $\varphi$ does not contain any subform similar
  to $\qform{1,-a,-b}$, we have to show that no three of the
  quadratic forms $\alpha$, $\beta$, $\gamma$, and $\delta$ have any
  common value over $F_1$. This can be checked after some scalar
  extension. For instance, by~\cite[Lemma(4.4)(iii)]{HVG}, the only
  common value of $\alpha$ and $\beta$ over $F_0(t)((u))$ is the
  square class of $u$. On the other hand,
  applying~\cite[VI(1.3)]{Lam}, one may check that $\gamma$ and
  $\delta$ are 
  both isomorphic over $F_0((t))((u))$ to $\qform{1,t}$, which does
  not represent $u$. On the other hand, over $F_0(t)((t+u))$ the form
  $\gamma$ only represents the square classes of $1$ and $t+u$,
  whereas $\delta$ only represents the square classes of $t$ and
  $t(1+t)(t+u)$. Therefore, $\gamma$ and $\delta$ have no common value
  in $F_1$.

  To complete the proof, we show $A$ is a division algebra. Taking the
  Clifford invariant of each side of \eqref{eq:3} and applying
  \cite[V(3.13)]{Lam}, we obtain the following equality in the Brauer
  group of $F$:
  \[
  [A]+[C_0(\qform{t(1+t), -at, -btu(t+u), ab, abu})]=[Q].
  \]
  The even Clifford algebra of the $5$-dimensional form is easily
  computed:
  \[
  C_0(\qform{t(1+t), -at, -btu(t+u), ab, abu})\simeq (-u,bt)_F\otimes
  (ab(t+u),-au(1+t))_F,
  \]
  hence
  \begin{align*}
    A & \simeq (a,b)_F\otimes (-u,bt)_F\otimes (ab(t+u),-au(1+t))_F\\
      & \simeq (-u,t)_F\otimes (a(t+u),u(1+t)(t+u))_F \otimes (b,1+t)_F.
  \end{align*}
  Since $b$ is a uniformizing parameter for the $b$-adic valuation on
  $F$, it follows from \cite[\S19.6, Prop.]{Pierce} that the right
  side is a division algebra 
  if (and only if) the algebra
  \[
  B=(-u,t)_{F_1((a))}\otimes (a(t+u),u(1+t)(t+u))_{F_1((a))}\otimes
  F_1\bigl(\sqrt{1+t}\bigr)((a))
  \]
  is division. (Alternatively, one may view $A$ as a ring of twisted
  Laurent series over $B$ in an indeterminate whose square is $b$.)
  Now, $a(t+u)$ is a uniformizing parameter for the 
  $a$-adic valuation on $F_1\bigl(\sqrt{1+t}\bigr)((a))$, hence the same
  argument shows that $B$ is a division algebra if (and only if) the
  algebra
  \[
  C=(-u,t)_{F_1}\otimes F_1\bigl(\sqrt{1+t}, \sqrt{u(1+t)(t+u)}\bigr)
  \]
  is division. Since $1+t$ and $u(t+u)$ are squares in $F_0(u)((t))$,
  we may embed $C$ in the quaternion algebra $(-u,t)_{F_0(u)((t))}$,
  which is clearly division. Therefore, $A$ is a division algebra.
\end{proof}

\section{Non-totally decomposable orthogonal involutions of degree 8}
\label{deg8.sec}

In this section, we consider the case of central simple algebras with
orthogonal involution $(A,\sigma)$ of degree~$8$ that are not totally
decomposable. These algebras do not contain any quaternion algebra
with canonical involution $(H,\binv)$, since the centralizer of $H$
would be an algebra of degree~$4$ with symplectic involution, hence
decomposable by \cite[(16.16)]{KMRT}; the algebra $(A,\sigma)$ would then be
totally decomposable.

We start with a couple of lemmas of independent interest related to
triality.  
Let $Q_1$, $Q_2$, $Q_3$ be quaternion $F$-algebras such that
$Q_1\otimes_FQ_2\otimes_FQ_3$ is split. By a well-known result due to
Albert and to Pfister, this condition implies that $Q_1$, $Q_2$, and
$Q_3$ have a common maximal subfield,
see~\cite[(16.30)]{KMRT}. Therefore, we may write 
\[
Q_1=(c,d_1)_F,\qquad Q_2=(c,d_2)_F,\qquad Q_3=(c,d_3)_F
\]
for some $c$, $d_1$, $d_2$, $d_3\in F^\times$ such that the quaternion
algebra $(c,d_1d_2d_3)_F$ is split. For $\alpha=1$, $2$, $3$, let
$\rho_\alpha$ be the orthogonal involution on $Q_\alpha$ with
$\disc\rho_\alpha=c$. The involution $\rho_\alpha$ is uniquely
determined up to conjugation by \cite[(7.4)]{KMRT}.

\begin{lem}
  \label{quat.lem}
  For $\{\alpha,\beta,\gamma\}=\{1,2,3\}$ we have
  \[
  (Q_\alpha,\binv)\otimes(Q_\beta,\binv)\simeq
  \Ad_{\pform{d_\alpha}}\otimes(Q_\gamma,\rho_\gamma) \simeq
  \Ad_{\pform{d_\beta}}\otimes(Q_\gamma,\rho_\gamma).
  \]
\end{lem}

\begin{proof}
  By Tao's computation of the Clifford algebra of a decomposable
  involution \cite{Tao:94} we have
  \begin{align*}
    C\bigl((Q_\alpha,\binv)\otimes(Q_\beta,\binv)\bigr)& \simeq
    Q_\alpha\times Q_\beta\simeq (c,d_\alpha)_F\times (c,d_\beta)_F,\\
    C\bigl(\Ad_{\pform{d_\alpha}}\otimes(Q_\gamma,\rho_\gamma)\bigr)&
    \simeq (c,d_\alpha)_F\times(c,d_\alpha d_\gamma)_F\simeq
    (c,d_\alpha)_F\times(c,d_\beta)_F,\\
    C\bigl(\Ad_{\pform{d_\beta}}\otimes(Q_\gamma,\rho_\gamma)\bigr)&
    \simeq (c,d_\beta)_F\times(c,d_\beta d_\gamma)_F\simeq
    (c,d_\beta)_F\times(c,d_\alpha)_F.
  \end{align*}
  Since central simple algebras with orthogonal involutions of
  degree~4 are classified by their Clifford algebra (see
  \cite[(15.7)]{KMRT}), the lemma follows.
\end{proof}

Now, for $\alpha=1$, $2$, $3$, let $(A_\alpha,\sigma_\alpha)$ be a
central simple $F$-algebra with orthogonal involution of degree~8 such
that for $\beta$, $\gamma$ with $\{\alpha,\beta,\gamma\}=\{1,2,3\}$,
\[
(A_\alpha,\sigma_\alpha)\simeq\Ad_{\qform{1,-1,1,-d_\beta}}\otimes
(Q_\alpha,\rho_\alpha) \simeq \Ad_{\qform{1,-1,1,-d_\gamma}}\otimes
(Q_\alpha,\rho_\alpha).
\]
Thus, $(A_\alpha,\sigma_\alpha)$ is Witt-equivalent to
$(Q_\beta,\binv)\otimes (Q_\gamma,\binv)$ by Lemma~\ref{quat.lem}.

\begin{lem}
  \label{tria.lem}
  The triple $\bigl((A_1,\sigma_1), (A_2,\sigma_2), (A_3,
  \sigma_3)\bigr)$ is trialitarian, in the sense that for
  $\{\alpha,\beta,\gamma\}=\{1,2,3\}$ we have
  \[
  C(A_\alpha,\sigma_\alpha)\simeq (A_\beta,\sigma_\beta) \otimes
  (A_\gamma,\sigma_\gamma).
  \]
  (See \cite[p.~548]{KMRT}.)
\end{lem}

\begin{proof}
  By triality, it suffices to prove the isomorphism for $\alpha=1$,
  $\beta=2$, and $\gamma=3$. By definition, $(A_1,\sigma_1)$ is an
  orthogonal sum of the algebra $M_2(Q_1)$ with a hyperbolic
  involution and of $\Ad_{\pform{d_2}}\otimes(Q_1,\rho_1)$, so by
  Lemma~\ref{quat.lem}
  \[
  (A_1,\sigma_1)\simeq \bigl((M_2(F),\binv)\otimes (Q_1,\binv)\bigr)
  \boxplus\bigl((Q_2,\binv)\otimes(Q_3,\binv)\bigr).
  \]
  By \cite[(15.12)]{KMRT} we have
  \[
  C\bigl((M_2(F),\binv)\otimes (Q_1,\binv)\bigr)\simeq (M_2(F),\binv)
  \times (Q_1,\binv)
  \]
  and
  \[
  C\bigl((Q_2,\binv)\otimes(Q_3,\binv)\bigr)\simeq (Q_2,\binv)\times
  (Q_3,\binv).
  \]
  Arguing as in Garibaldi's ``Orthogonal Sum Lemma''
  \cite[Lemma~3.2]{G:clif}, we get
  \[
  (C(A_1,\sigma_1),\underline{\sigma_1})\simeq
  (C_+(A_1,\sigma_1),\sigma_+) \times (C_-(A_1,\sigma_1),\sigma_-)
  \]
  with
  \[
  (C_+(A_1,\sigma_1),\sigma_+)\simeq
  \bigl((M_2(F),\binv)\otimes(Q_2,\binv)\bigr) \boxplus
  \bigl((Q_1,\binv)\otimes(Q_3,\binv)\bigr)
  \]
  and
  \[
  (C_-(A_1,\sigma_1),\sigma_-)\simeq
  \bigl((M_2(F),\binv)\otimes(Q_3,\binv)\bigr) \boxplus
  \bigl((Q_1,\binv)\otimes(Q_2,\binv)\bigr).
  \]
  Thus, $(C_+(A_1,\sigma_1),\sigma_+)$ is Witt-equivalent to
  $(Q_1,\binv)\otimes(Q_3,\binv)$, hence it is isomorphic to
  $(A_2,\sigma_2)$. Likewise, $(C_-(A_1,\sigma_1),\sigma_-)$ is
  isomorphic to $(A_3,\sigma_3)$.
\end{proof}

\begin{thm}
  \label{deg8.thm}
  Let $(A,\sigma)$ be a central simple $F$-algebra with orthogonal
  involution of degree~$8$. Assume $(A,\sigma)$ is not totally
  decomposable. Then $(A,\sigma)_{F_Q}$ is hyperbolic if and only if
  there is a quaternion $F$-algebra $Q'$ with the following
  properties:
  \begin{itemize}
  \item[--]
  $\ind(Q\otimes_FQ')\leq2$, and
  \item[--]
  $(A,\sigma)$ is Witt-equivalent to $(Q,\binv)\otimes(Q',\binv)$.
  \end{itemize}
  When these equivalent properties hold, we can find $c$, $d$, $d'\in
  F^\times$ such that
  \[
  Q\simeq(c,d)_F,\qquad Q'\simeq(c,d')_F,
  \]
  and
  \[
  (A,\sigma)\simeq\Ad_{\qform{1,-1,1,-d}}\otimes(Q'',\rho'')
  \]
  where $Q''=(c,dd')_F$ and $\rho''$ is an orthogonal involution on
  $Q''$ with $\disc\rho''=c$.
\end{thm}

\begin{proof}
  Clearly, $(A,\sigma)_{F_Q}$ is hyperbolic if $(A,\sigma)$ is
  Witt-equivalent to an algebra containing $\Qc$. Conversely, suppose
  $(A,\sigma)_{F_Q}$ is hyperbolic. Since $(A,\sigma)$ is not totally
  decomposable, it is not hyperbolic. If $A$ is split,
  Proposition~\ref{split.prop} shows that the anisotropic kernel of
  $(A,\sigma)$ is $\Ad_{n_Q}\simeq\Qc\otimes\Qc$, hence
  \[
  (A,\sigma)\simeq\Ad_{\qform{1,-1,1,-a}}\otimes\Ad_{\pform{b}}.
  \]
  For the rest of the proof, we may thus assume $A$ is not split. By
  Proposition~\ref{orthog.prop}, one of the components of the Clifford
  algebra, $C_+(A,\sigma)$ say, is split by $F_Q$. However,
  $C_+(A,\sigma)$ is not split since $(A,\sigma)$ is not totally
  decomposable, hence $C_+(A,\sigma)$ is Brauer-equivalent to $Q$. As
  was observed in Proposition~\ref{orthog.prop},
  $(C_+(A,\sigma),\sigma_+)_{F_Q}$ is isotropic. If it is hyperbolic,
  then $(A,\sigma)$ is hyperbolic by the main theorem of \cite{G:clif}, a
  contradiction. Therefore, the anisotropic kernel of
  $(C_+(A,\sigma),\sigma_+)$ has degree~$4$. It has discriminant~$1$
  by triality, hence $(C_+(A,\sigma),\sigma_+)$ is Witt-equivalent to
  a product $(Q',\binv)\otimes(Q'',\binv)$ for some quaternion
  $F$-algebras $Q'$, $Q''$ such that $Q'\otimes Q''$ is
  Brauer-equivalent to $Q$. We may therefore find $c$, $d$, $d'$,
  $d''\in F^\times$ such that
  \[
  Q\simeq(c,d)_F,\qquad Q'\simeq(c,d')_F,\qquad Q''\simeq(c,d'')_F.
  \]
  Letting $\rho$ (resp.\ $\rho'$, resp.\ $\rho''$) be an orthogonal
  involution on $Q$ (resp.\ $Q'$, resp.\ $Q''$) with discriminant $c$,
  we have by Lemma~\ref{quat.lem}
  \[
  (C_+(A,\sigma),\sigma_+)\simeq \Ad_{\qform{1,-1,1,-d'}}\otimes
  (Q,\rho)\simeq \Ad_{\qform{1,-1,1,-d''}}\otimes(Q,\rho).
  \]
  By Lemma~\ref{tria.lem}, it follows that $(A,\sigma)$ is isomorphic
  to
  \[
  \Ad_{\qform{1,-1,1,-d}}\otimes(Q',\rho') \qquad\text{or}\qquad
  \Ad_{\qform{1,-1,1,-d}}\otimes(Q'',\rho''),
  \]
  hence it is Witt-equivalent to
  \[
  \Qc\otimes(Q',\binv) \qquad\text{or}\qquad
  \Qc\otimes(Q'',\binv).
  \]
  Interchanging $Q'$ and $Q''$ if necessary, we thus obtain the stated
  description of $(A,\sigma)$.
\end{proof}

\section{Examples of arbitrarily large degree}
\label{large.sec}

Let $(A,\sigma)$ be a central simple $F$-algebra with involution of
orthogonal or symplectic type. Consider the (iterated) Laurent series
fields $F_1=F((x))$, $F_2=F((x))((y))$, and the quaternion
$F_2$-algebra $H=(x,y)_{F_2}$. Let $\rho$ be any involution of
orthogonal or symplectic type on $H$, and let
\[
(A_1,\sigma_1)=(A,\sigma)\otimes_F\Ad_{\pform{x}},\qquad
(A_2,\sigma_2)=(A,\sigma)\otimes_F (H,\rho).
\]
If $(A,\sigma)_{F_Q}$ is hyperbolic, then $(A_1,\sigma_1)_{F_Q}$ and
$(A_2,\sigma_2)_{F_Q}$ also are hyperbolic, since they contain a
hyperbolic factor.

\begin{thm}
  \label{large.thm}
  Assume $(A,\sigma)$ is anisotropic. Then $(A_1,\sigma_1)$ and
  $(A_2,\sigma_2)$ are anisotropic. Moreover, the following conditions
  are equivalent:
  \begin{enumerate}
  \item[(i)]
  $(A,\sigma)$ contains $\Qc$;
  \item[(ii)]
  $(A_1,\sigma_1)$ contains $\Qc$;
  \item[(iii)]
  $(A_2,\sigma_2)$ contains $\Qc$.
  \end{enumerate}
\end{thm}

\begin{proof}
  Let $\xi_1=\bigl(
  \begin{smallmatrix}
    1&0\\0&-1
  \end{smallmatrix}\bigr)$, $\eta_1=\bigl(
  \begin{smallmatrix}
    0&x\\1&0
  \end{smallmatrix}\bigr)\in M_2(F_1)$, so 
  \[
  \ad_{\pform{x}}(\xi_1)=\xi_1 \qquad\text{and}\qquad
  \ad_{\pform{x}}(\eta_1)=-\eta_1.
  \]
  Let also $(a_i)_{i\in I}$ be an $F$-basis of $A$, so $(a_i\otimes1,
  a_i\otimes\xi_1, a_i\otimes\eta_1, a_i\otimes\xi_1\eta_1)_{i\in I}$
  is an $F_1$-basis of $A_1$. We extend the $x$-adic valuation $v_1$
  on $F_1$ to a map
  \[
  g_1\colon
  A_1\to\bigl({\textstyle\frac12}\mathbb{Z}\bigr)\cup\{\infty\}
  \]
  defined by
  \[
  g_1\bigl(\sum_{i\in I} a_i\otimes(\alpha_i+\beta_i\xi_1
  +\gamma_i\eta_1+\delta_i\xi_1\eta_1)\bigr) =\min_{i\in I}\bigl(
  v_1(\alpha_i), v_1(\beta_i), v_1(\gamma_i)+{\textstyle\frac12},
  v_1(\delta_i)+{\textstyle\frac12}\bigr)
  \]
  for $\alpha_i$, $\beta_i$, $\gamma_i$, $\delta_i\in F_1$. It is
  readily verified that the map $g_1$ satisfies the following
  conditions for $s$, $t\in A_1$ and $\alpha\in F_1$:
  \begin{itemize}
  \item
  $g_1(1)=0$ and $g_1(s)=\infty$ if and only if $s=0$;
  \item
  $g_1(s+t)\geq\min\bigl(g_1(s), g_1(t)\bigr)$ and
  $g_1(s\alpha)=g_1(s)+v_1(\alpha)$;
  \item
  $g_1(st)\geq g_1(s)+g_1(t)$.
  \end{itemize}
  (It suffices to prove the last inequality for $s$, $t$ in the above
  $F_1$-base of $A_1$, see~\cite[Lemma~1.2]{TW1}.) The map $g_1$ defines a
  filtration of $A_1$, and the associated graded ring $\gr(A_1)$ is
  \[
  \gr(A_1)\simeq A\otimes_F M_2(F[x,x^{-1}])
  \]
  with the grading defined by
  \begin{align*}
    \gr(A_1)_\lambda & = A\otimes
    \begin{pmatrix}
    x^\lambda&0\\0&x^\lambda
    \end{pmatrix}
    &&\text{for $\lambda\in\mathbb{Z}$},\\
    \gr(A_1)_\lambda & = A\otimes
    \begin{pmatrix}
    0&x^{\lambda+\frac12}\\ x^{\lambda-\frac12}&0
    \end{pmatrix}
    &&\text{for $\lambda\in\bigl({\textstyle\frac12}\mathbb{Z}\bigr)
      \setminus\mathbb{Z}$.} 
  \end{align*}
  Therefore, $\gr(A_1)$ is a graded simple algebra, and $g_1$ is a
  $v_1$-gauge in the sense of \cite{TW1}. The involution $\sigma_1$
  preserves $g_1$. On $\gr(A_1)_0$, the induced involution
  $\widetilde{\sigma_1}$ is $\sigma\otimes\Id$, hence it is
  anisotropic. Therefore, $\sigma_1$ is anisotropic by \cite[Cor.~2.3]{TW2},
  $g_1$ is the unique $v_1$-gauge that is preserved by $\sigma_1$ by
  \cite[Th.~2.2]{TW2}, and we have
  \begin{equation}
    \label{eq:4}
    g_1(\sigma_1(s)s)=2g_1(s)\qquad\text{for all $s\in A_1$}.
  \end{equation}
  Now, suppose $(A_1,\sigma_1)$ contains $\Qc$; it then contains
  elements $i$, $j$ such that
  \begin{equation}
    \label{eq:5}
    i^2=a,\quad j^2=b,\quad ji=-ij,\quad\sigma_1(i)=-i,\quad
    \sigma_1(j)=-j. 
  \end{equation}
  Then by \eqref{eq:4} we have $g_1(i)=\frac12g_1(-a)=0$ and,
  similarly, $g_1(j)=0$. The images $\widetilde{i}$, $\widetilde{j}$ of
  $i$, $j$ in $\gr_1(A_1)_0$ satisfy conditions similar
  to~\eqref{eq:5}. Since $\gr(A_1)_0\simeq A\times A$ we may consider
  a projection $\gr(A)_0\to A$, which is a homomorphism of algebras
  with involution $\pi\colon(\gr(A_1)_0,\widetilde{\sigma_1})\to
  (A,\sigma)$. The images $\pi(\widetilde{i})$, $\pi(\widetilde{j})$
  generate a copy of $\Qc$ in $(A,\sigma)$. Thus, $(A,\sigma)$
  contains $\Qc$ if $(A_1,\sigma_1)$ contains $\Qc$. The converse is
  clear.

  The argument for $(A_2,\sigma_2)$ follows the same lines. Let
  $\xi_2$, $\eta_2\in H$ be such that
  \[
  \xi_2^2=x,\qquad \eta_2^2=y,\qquad\eta_2\xi_2=-\xi_2\eta_2.
  \]
  Note that if $\rho$ is orthogonal its discriminant is represented by
  the quadratic form $\qform{x,y,-xy}$, hence it is the square class
  of $x$, $y$, or $-xy$. Therefore, we may assume
  $\rho=\Int(\xi_2)\circ\binv$, $\Int(\eta_2)\circ\binv$, or
  $\Int(\xi_2\eta_2)\circ\binv$. In each case (and also if $\rho$ is
  symplectic) we have $\rho(\xi_2)=\pm\xi_2$ and
  $\rho(\eta_2)=\pm\eta_2$.

  Let $v_2\colon F_2\to \mathbb{Z}^2\cup\{\infty\}$ be the
  $(x,y)$-adic valuation such that $v_2(x^\lambda
  y^\mu)=(\lambda,\mu)$ for $\lambda$, $\mu\in \mathbb{Z}$, where
  $\mathbb{Z}^2$ is endowed with the right-to-left lexicographic
  ordering. Considering again an $F$-basis $(a_i)_{i\in I}$ of $A$, we
  extend $v_2$ to a map
  \[
  g_2\colon
  A_2\to\bigl({\textstyle\frac12}\mathbb{Z}\bigr)^2\cup\{\infty\}
  \]
  defined by
  \begin{multline*}
  g_2\bigl(\sum_{i\in I} a_i\otimes(\alpha_i+\beta_i\xi_2
  +\gamma_i\eta_2+\delta_i\xi_2\eta_2)\bigr) =\\ \min_{i\in
    I}\bigl(v_2(\alpha_i), v_2(\beta_i)+({\textstyle\frac12},0),
  v_2(\gamma_i)+(0,{\textstyle\frac12}),
  v_2(\delta_i)+({\textstyle\frac12,\frac12})\bigr)
  \end{multline*}
  for $\alpha_i$, $\beta_i$, $\gamma_i$, $\delta_i\in F_2$. The map
  $g_2$ is a $v_2$-gauge on $A_2$ with associated graded ring
  \[
  \gr(A_2)=A\otimes(x,y)_{F[x,x^{-1},y,y^{-1}]}.
  \]
  The involution $\sigma_2$ preserves $g_2$ and the induced involution
  $\widetilde{\sigma_2}$ on $\gr(A_2)_0=A$ is $\sigma$. Therefore, the
  same arguments as for $(A_1,\sigma_1)$ show that $(A_2,\sigma_2)$ is
  anisotropic, and that $(A_2,\sigma_2)$ contains $\Qc$ if and only if
  $(A,\sigma)$ contains $\Qc$.
\end{proof}

Theorem~\ref{large.thm} applies in particular to the division algebra
with orthogonal involution $(A,\sigma)$ of Theorem~\ref{totdec.thm},
and yields central simple algebras with anisotropic involution
$(A_1,\sigma_1)$ and $(A_2,\sigma_2)$ of degree~$16$  that do not
contain $\Qc$, even though they are hyperbolic over $F_Q$. The
involution $\sigma_1$ is orthogonal and $\ind A_1=8$, while the
involution $\sigma_2$ may be of orthogonal or symplectic type and
$A_2$ is division. Of course, these constructions can be iterated to
obtain examples of algebras with anisotropic involution of arbitrarily
large degree that become hyperbolic over $F_Q$ and do not contain
$\Qc$. Such examples can also be derived from the central simple
algebras of degree~$4$ with symplectic involution in case~(b) of
Theorem~\ref{deg4symp.thm}, although no division algebra can be
obtained in this way since the algebras in case~(b) of
Theorem~\ref{deg4symp.thm} have index~$2$.

\providecommand{\bysame}{\leavevmode\hbox to3em{\hrulefill}\thinspace}

\end{document}